\documentclass[]{article}
\usepackage{tikz}
\usepackage{caption}
\usetikzlibrary{matrix}

\usepackage{color}

\usepackage{graphicx}        % standard LaTeX graphics tool
\usepackage{amsmath,amssymb}
\usepackage{amsthm,amscd}
\newtheorem{theorem}{Theorem}[section]
\newtheorem{proposition}[theorem]{Proposition}
\newtheorem{lemma}[theorem]{Lemma}
\newtheorem{definition}[theorem]{Definition}
\newtheorem{example}[theorem]{Example}
\newtheorem{property}[theorem]{Property}

\newtheorem{corollary}[theorem]{Corollary}

\newcommand{\la}{\langle}
\newcommand{\ra}{\rangle}
\newcommand{\xad}{x_\alpha^\delta}
\newcommand{\xdag}{x^\dagger}

\title{On $\mathbf{\ell^1}$-regularization under continuity of the forward operator in weaker topologies}
\author{Daniel Gerth$^*$ and Bernd Hofmann\footnote{Chemnitz University of Technology, Faculty of Mathematics, 09107 Chemnitz, Germany,  \{daniel.gerth,bernd.hofmann\}@mathematik.tu-chemnitz.de}}

\begin{document}

\maketitle

\abstract{Our focus is on the stable approximate solution of linear operator equations based on noisy data by using  $\ell^1$-regularization as a sparsity-enforcing version of Tikhonov regularization. We summarize recent results on situations where the sparsity of the solution slightly fails. In particular, we show how the recently established theory for weak*-to-weak continuous linear forward operators can be extended to the case of weak*-to-weak* continuity. This might be of interest when the image space is non-reflexive. We discuss existence, stability and convergence of regularized solutions. For injective operators, we will formulate convergence rates by exploiting variational source conditions. The typical rate function obtained under an ill-posed operator is strictly concave and the degree of failure of the solution sparsity has an impact on its behavior. Linear convergence rates just occur in the two borderline cases of proper sparsity, where the solutions belong to $\ell^0$, and of well-posedness. For an exemplary operator, we demonstrate that the technical properties used in our theory can be verified in practice. In the last section, we briefly mention the difficult case of oversmoothing regularization where $x^\dag$ does not belong to $\ell^1$.}

\section{Introduction}
\label{sec:1}

We are going to deal with the stable solution of linear operator equations
\begin{equation} \label{eq:opeq}
A x =y
\end{equation}
with a {\it bounded linear} operator $A:\ell^1 \to Y$, mapping from the {\it non-reflexive} infinite dimensional space $\ell^1$ of absolutely summable infinite real or complex sequences to an {\it infinite dimensional Banach space} $Y$. Instead of the exact right-hand side $y$ from the range $\mathcal{R}(A)$ of $A$ we assume to have only noisy data $y^\delta \in Y$ available which satisfy the deterministic noise model
\begin{equation} \label{eq:noise}
\|y-y^\delta\|_Y \le \delta
\end{equation}
with prescribed noise level $\delta>0$. Our focus for solving equation \eqref{eq:opeq} is on the method of $\ell^1$-regularization,
where for regularization parameters $\alpha>0$ the minimizers $\xad$ of the extremal problem
\begin{equation} \label{eq:TIK}
\frac{1}{p}\|A x - y^\delta\|_Y^p + \alpha\,\|x\|_{\ell^1} \to \min, \qquad \mbox{subject to}\quad x \in \ell^1,
\end{equation}
are used as approximate solutions. This method is a sparsity-enforcing version of Tikhonov regularization, possessing applications in different branches of imaging, natural sciences, engineering and mathematical finance. It was
comprehensively analyzed with all its facets and varieties in the last fifteen years (cf., e.g.,~the corresponding chapters in the books \cite{Scherzered11,Scherzetal09,Schusterbuch12} and the papers \cite{AnzHofRam13,BreLor09,Daub04,FleHeg14,Grasm10,Grasmei08,Grasmei11,Lorenz08,Ramlau08,RamRes10}). We restrict our considerations to {\it injective} operators $A$
such that the element $\xdag=(\xdag_1,\xdag_2,...) \in \ell^1$ denotes the uniquely determined solution to \eqref{eq:opeq}. For assertions concerning the case of non-injective
operators $A$ in the context of $\ell^1$-regularization, we refer to \cite{Flemming16}. In the non-injective case, even the $\ell^1$-norm minimizing solutions need not be uniquely determined. As a consequence, very technical
conditions must be introduced in order to formulate convergence assertions and rates. In our framework, the Propositions \ref{prop:weakstarweak} and \ref{prop:weakstarweakstar} below would have to be adapted, which however is out of the scope of this paper.

With the paper \cite{BurFleHof13} as starting point and preferably based on variational source conditions first introduced in
\cite{HKPS07}, convergence rates for $\ell^1$-regularization of operator equations \eqref{eq:opeq} and variants like elastic-net
 \begin{equation} \label{eq:NET}
\frac{1}{p}\|A x - y^\delta\|_Y^p+\alpha\left(\frac{1}{2}\|x\|_{\ell^2}+\eta\|x\|_{\ell^1}\right) \to \min, \mbox{ subject to }  x \in \ell^1
\end{equation}
have been verified under the condition that the sparsity assumption slightly fails (cf.~\cite{CHZ17,FleHofVes15,FleHofVes16}). This means that the solution $\xdag \in \ell^1$ is not sparse, abbreviated as $\xdag \notin \ell^0$.
Most recently in \cite{FleGer17}, the first author and Jens Flemming have shown that complicated conditions on $A$, usually supposed for proving convergence rates in $\ell^1$-regularization (cf.~\cite[Assumption~2.2~(c)]{BurFleHof13} and condition \eqref{eq:BFH} below), can be simplified to the requirement of weak$^*$-to-weak continuity of the injective operator $A$. This seems to be convincing if $Y$ is a reflexive Banach space. The present paper, however, makes assertions also in the case that $A$ is only weak$^*$-to-weak$^*$ continuous, which is of interest for
non-reflexive Banach spaces $Y$. Moreover, we complement results from \cite{FleGer17}, for example with respect to the well-posed situation.

The paper is organized as follows. In Section \ref{sec:2} we recall basic properties of $\ell^1$-regularization. We proceed in Section \ref{sec:3} by discussing the ill-posedness of equation \eqref{eq:opeq}. We mention that in particular variational source conditions  allow us to deal with the ill-posedness and yield convergence rates. For our convergence analysis a particular property of the operator is necessary. In Section \ref{sec:4} we show that weak*-to-weak continuity and injectivity imply this property. Interestingly, the same property holds under weak*-to-weak* continuity and injectivity as shown in Section \ref{sec:5}. There we also derive the convergence rates which hold for both continuity assumptions. Finally, we demonstrate that even the case of a well-posed operator is reflected in our property in Section \ref{sec:6}. There we also hint at the case of oversmoothing regularization, which occurs when one employs $\ell^1$-regularization although the true solution $x^\dag$ does not belong to $\ell^1$.

\section{Preliminaries and basic propositions}
\label{sec:2}

In this paper, we consider the variant \eqref{eq:TIK} of $\ell^1$-regularization with some exponent $p>1$ and with a regularization parameter $\alpha>0$. Let $y \in \mathcal{R}(A)$. Then, due to the injectivity of $A$, there exists a uniquely determined solution $\xdag \in \ell^1$ to \eqref{eq:opeq}. With the following Proposition~\ref{pro:prorep} we recall the assertions of Proposition~2.8 in \cite{BurFleHof13}
with respect to existence, stability, convergence and sparsity of the $\ell^1$-regularized solutions $\xad$. The proof ibidem emphasizes the fact that most of these properties follow directly from the general theory of Tikhonov regularization in Banach spaces (cf., e.g., \cite[Section~3]{HKPS07} and \cite[Section~4.1]{Schusterbuch12}). Since for $p>1$ the Tikhonov functional to be minimized in \eqref{eq:TIK} is strictly convex, the regularized solutions
$\xad$, whenever they exist, are uniquely determined for all $\alpha>0$.

\begin{proposition} \label{pro:prorep}
Let $A: \ell^1 \to Y$ be weak$^*$-to-weak continuous, i.e., $x_n \rightharpoonup^* x_0$ in $\ell^1$ implies that $Ax_n \rightharpoonup Ax_0$ in $Y$.
Then for all $\alpha>0$ and all $y^\delta \in Y$ there exist uniquely determined minimizers $\xad \in \ell^1$ of the Tikhonov functional from \eqref{eq:TIK}. These regularized solutions are sparse, i.e., $\xad \in  \ell^0$, and they are stable with respect to the data, i.e., small perturbations in $y^\delta$ in the norm topology of $Y$ lead only to small changes in $\xad$ with respect to the weak$^*$-topology in $\ell^1$.  If $\delta_n \to 0$ and if the regularization parameters $\alpha_n=\alpha(\delta_n,y^{\delta_n})$ are chosen such that $\alpha_n \to 0$ and $\frac{\delta_n^p}{\alpha_n}\to 0$ as $n \to \infty$, then  $x_{\alpha_n}^{\delta_n}$ converges in the weak$^*$-topology of $\ell^1$ to the
uniquely determined solution $\xdag$ of the operator equation \eqref{eq:opeq}. Moreover we have $\lim \limits_{n \to \infty} \|x_{\alpha_n}^{\delta_n}\|_{\ell^1}=\|\xdag\|_{\ell^1}$, which, as a consequence of the weak$^*$ Kadec-Klee property in $\ell^1$ (see, e.g.,~\cite[Lemma~2.2]{BoHo13}), implies norm convergence
$$\lim \limits_{n \to \infty} \|x_{\alpha_n}^{\delta_n}-\xdag\|_{\ell^1}=0.$$
\end{proposition}

The weak$^*$-to-weak continuity of $A$ in combination with the {\it stabilizing property} of the penalty functional $\|x\|_{\ell^1}$ in $\ell^1$ together with an appropriate choice of the regularization parameter $\alpha>0$ represent basic assumptions of Proposition~\ref{pro:prorep}. In contrast to regularization in reflexive Banach space, where the level sets of the norm functional are weakly compact, we have in $\ell^1$ weak$^*$ compactness of the corresponding level sets
according to the sequential Banach-Alaoglu theorem (cf., e.g., \cite[Theorems~3.15 and 3.17]{Rudin91}), which we present in form of the following lemma.

\begin{lemma} \label{lem:Alaoglu}
The closed unit ball of a Banach space $X$ is compact in the weak$^*$-topology if there is a separable Banach space $Z$ (predual space) with dual $Z^*=X$. Then any bounded sequence $\{x_n\}_{n \in \mathbb{N}}$ in $X$  has a weak$^*$-convergent subsequence $\{x_{n_k}\}_{k \in \mathbb{N}}$ such that $x_{n_k} \rightharpoonup^* x_0 \in X$ as $k \to \infty$.
\end{lemma}
The occurring kind of compactness of the level sets with $X=\ell^1$ and predual space $Z=c_0$ ensures the existence of minimizers $\xad$ of the functional \eqref{eq:TIK}.

Throughout this paper, we use the terms `continuous', `compact' or `lower semicontinuous' for an operator, a set or a functional always in the sense of `sequentially continuous', `sequentially compact' or `sequentially lower semicontinuous'. As the Lemmas~6.3 and 6.5 from \cite{FlemmingHabil17} show, there is no reason for a distinction in case of using weak topologies. From Lemma~2.7 and Proposition~2.4 in \cite{BurFleHof13} one can take assertions concerning sufficient conditions for
the weak$^*$-to-weak continuity of $A$, which we summarize in the Proposition~\ref{lem:lemrep} below.
As also indicated in Proposition~\ref{pro:prorep}, for the choice of $\alpha$, the so-called regularization property
\begin{equation} \label{eq:regprop}
\alpha(\delta,y^\delta) \to 0 \qquad \mbox{and} \qquad \frac{\delta^p}{\alpha(\delta,y^\delta)}\to 0 \qquad \mbox{as}\quad \delta \to 0,
\end{equation}
where $\alpha$ tends to zero, but sufficiently slow, plays an important role. In our studies, we consider on the one hand {\it a priori parameter choices} $\alpha_{APRI}=\alpha(\delta)$ defined as
\begin{equation} \label{eq:alphaapriori}
\alpha(\delta):= \frac{\delta^p}{\varphi(\delta)}, \qquad 0 <\delta \le \overline{\delta},
\end{equation}
with {\it concave} index functions $\varphi$. We call $\varphi: [0,\infty) \to [0,\infty)$ an {\it index function} if $\varphi$ with $\varphi(0)=0$ is continuous and strictly increasing. Obviously,
an a priori parameter choice $\alpha_{APRI}$ from \eqref{eq:alphaapriori} with an arbitrary concave index function $\varphi$ satisfies \eqref{eq:regprop} as $\lim _{\delta \to +0} \varphi(\delta)=0$ is valid for each index function and we have $\frac{\delta^p}{\varphi(\delta)}=
\frac{\delta}{\varphi(\delta)}\,\delta^{p-1} \to 0$ as $\delta \to 0$, because $\delta^{p-1}$ is an index function for all exponents $p>1$ in \eqref{eq:TIK} and the factor $\frac{\delta}{\varphi(\delta)}$ is
bounded whenever $\varphi$ is concave.

On the other hand, we consider the {\it sequential discrepancy principle}, comprehensively analyzed in \cite{AnzHofMat14} (see also \cite{HofMat12}), as a specific {\it a posteriori parameter choice} $\alpha_{SDP}
=\alpha(\delta,y^\delta)$ for the regularization parameter. For prescribed $\tau>1,\;0 < q < 1,$ and a sufficiently large value $\alpha_0>0$, we let
  \[ \Delta_q := \{ \alpha_j>0:\, \alpha_j = q^j \alpha_0, \quad j=1,2,... \}. \]
Given $\delta >0$  and $y^\delta \in Y$, we choose $\alpha=\alpha_{SDP} \in\Delta_{q}$ according to the sequential discrepancy principle such that
  \begin{equation}\label{eq:sdp}
    \| A x_{\alpha}^\delta - y^\delta\| \leq \tau \delta < \|A {x_{\alpha/q}^{\delta}} - y^\delta\|.
  \end{equation}
By Theorem~1 in \cite{AnzHofMat14} it has been shown that there is $\overline{\delta}>0$ such that $\alpha_{SDP}$ is well-defined for $0<\delta \le \overline{\delta}$ and satisfies \eqref{eq:regprop} whenever
data compatibility in the sense of \cite[Assumption~3]{AnzHofMat14} takes place.

Consequently, both regularization parameter choices $\alpha=\alpha_{APRI}$ and $\alpha=\alpha_{SDP}$ are applicable for the $\ell^1$-regularization in order to get existence, stability and convergence of regularized
solutions in the sense of Proposition~\ref{pro:prorep}.
Now we are going to discuss conditions under which weak$^*$-to-weak continuity of $A: \ell^1 \to Y$ can be obtained. The occurring cross connections are relevant in order to ensure existence, stability and convergence of regularized solutions, but they have also an essential impact on convergence rates which will be discussed in Section~\ref{sec:4}.

\begin{proposition} \label{lem:lemrep}
Let $A: \ell^1 \to Y$ with adjoint operator $A^*: Y^* \to \ell^\infty$ satisfy the condition
\begin{equation} \label{eq:rangeinc0}
\mathcal{R}(A^*) \subseteq c_0,
\end{equation}
where $c_0$ is the Banach space of real-valued sequences converging to zero equipped with the supremum norm.
Then $A$ is weak$^*$-to-weak continuous.
In particular, \eqref{eq:rangeinc0} is fulfilled whenever there exist, for all $k \in \mathbb{N}$, source elements $f^{(k)} \in Y^*$ such that the system of source conditions
\begin{equation} \label{eq:BFH}
e^{(k)}= A^* f^{(k)}
\end{equation}
holds true, where $\{e^{(k)}\}_{k \in \mathbb{N}}$ is the sequence of $k$-th standard unit vectors which forms a Schauder basis in $c_0$. Under the condition \eqref{eq:BFH}
we even have the equality
\begin{equation} \label{eq:equal}
\overline{\mathcal{R}(A^*)}^{\ell^\infty} = c_0.
\end{equation}
\end{proposition}

The paper \cite{AnzHofRam13} shows that the condition \eqref{eq:BFH}, originally introduced by Grasmair in \cite{Grasm09}, can be verified for a wide class of applied linear inverse problems. But as also
the counterexamples in \cite{FleHeg14} indicate, it may fail if the underlying basis smoothness is insufficient.
However, weak$^*$-to-weak continuity of $A$ can be reformulated in several ways as the following proposition, proven in \cite[Lemma~2.1]{Flemming16}, shows. This proposition brings more order into the system of conditions.

\begin{proposition} \label{pro:flid1}
The three assertions
\begin{itemize}
\item[(i)]
\quad$ \{Ae^{(k)}\}_{k\in\mathbb{N}}$ converges in $Y$ weakly to zero, i.e.~$Ae^{(k)} \overset{k \to \infty}{\rightharpoonup} 0$\,,
\item[(ii)]
\quad$\mathcal{R}(A^\ast)\subseteq c_0$\,,
\item[(iii)]
\quad$A$ is weak$^*$-to-weak continuous\,,
\end{itemize}
are equivalent.
\end{proposition}

As outlined in \cite{BurFleHof13}, the operator equation \eqref{eq:opeq} with operator $A: \ell^1 \to Y$ is often motivated by a background operator equation $\tilde A \tilde x =y$ with an injective and bounded linear operator $\tilde A$ mapping from the infinite dimensional Banach space $\tilde X$ with uniformly bounded Schauder basis $\{u^{(k)}\}_{k \in \mathbb{N}}$, i.e.~$\|u^{(k)}\|_{\tilde X} \le K<\infty$, to the Banach space $Y$. Here, following the setting in \cite{Grasm09} we take into account a {\it synthesis operator} $L:\ell^1 \to \tilde X$ defined as $Lx:=\sum \limits_{k=1}^\infty x_k u^{(k)}$ for $x=(x_1,x_2,...) \in \ell^1$, which is a well-defined, injective and
bounded linear operator, and so is the composite operator $A=\tilde A \circ L: \ell^1 \to Y$. In particular $A$ is always weak$^*$-to-weak continuous if $A$ has a bounded extension to $\ell^p$, $1<p<\infty$, as this yields (i) in Proposition \ref{pro:flid1}. Even more specific, $A$ is weak*-to-weak continuous if $\tilde X$ is a Hilbert space. Since this case appears rather often in practice, the continuity property comes ``for free" in this situation.

\section{Ill-posedness and conditional stability} \label{sec:3}

In this section, we discuss ill-posedness phenomena of the operator equation \eqref{eq:opeq} based on Nashed's definition from \cite{Nashed86}, which we formulate in the following as Definition~\ref{def:Nashed} for the simplified case
of an injective bounded linear operator. Moreover, we draw a connecting line to the phenomenon of conditional well-posedness characterized by conditional stability estimates, which yield for appropriate choices of the regularization parameter convergence rates in Tikhonov-type regularization.

\begin{definition} \label{def:Nashed}
The operator equation $Ax=y$ with an injective bounded linear operator $A: X \to Y$ mapping between infinite dimensional Banach spaces $X$ and $Y$ is called {\it well-posed} if the range $\mathcal{R}(A)$ of $A$ is a closed subset of $Y$, otherwise the equation is called {\it ill-posed}. In the ill-posed case, we call the equation {\it ill-posed of type~I} if $\mathcal{R}(A)$ contains an infinite dimensional closed subspace and otherwise
{\it ill-posed of type~II}.
\end{definition}

The following proposition taken from \cite[Proposition~4.2 and 4.4]{FleHofVes15} and the associated Figure~1 gives some more insight into the different situations distinguished in Definition~\ref{def:Nashed}.

\begin{proposition} \label{pro:FHV1}
Consider the operator equation $Ax=y$ from Definition~\ref{def:Nashed}. If this equation is well-posed, i.e., $\mathcal{R}(A)=\overline{\mathcal{R}(A)}^Y$ and there is some constant $\underline c>0$ such that $\|Ax\| \ge \underline{c}\,\|x\| $ for all $x \in X\,$
or the equation is ill-posed of type~I, then the operator $A$ is non-compact. Consequently, compactness of $A$ implies ill-posedness of type~II.
More precisely, for an ill-posed equation $Ax=y$ with injective $A$ and infinite dimensional Banach spaces $X$ and $Y$, ill-posedness of type~II occurs if and only if $A$ is strictly singular. This means that the restriction of $A$ to an infinite dimensional subspace of $X$ is never an isomorphism (linear homeomorphism). If $X$ and $Y$ are both Hilbert spaces and the equation is ill-posed, then ill-posedness of type~II occurs if and only if $A$ is compact.
\end{proposition}

\vspace{-0.4cm}

\begin{figure}[h] \label{fig:FHV1}
\centerline{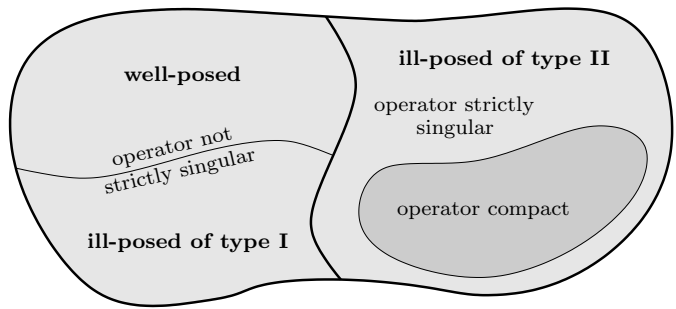}
\caption{Properties of $A$ for well-posedness and ill-posedness types of equations from Definition~\ref{def:Nashed}.}
\end{figure}

Now we apply the case distinction of Definition~\ref{def:Nashed}, verified in detail in Proposition~\ref{pro:FHV1}, to our situation of equation \eqref{eq:opeq} with $X:=\ell^1$ and $A:\ell^1 \to Y$. We start with a general
observation in Proposition~\ref{pro:notI}, which motivates the use of $\ell^1$-regularization for the stable approximate solution of \eqref{eq:opeq}, because the equation is mostly ill-posed. Below we enlighten the cross connections a bit more by the discussion of some example situations.

\begin{proposition} \label{pro:notI}
If $Y$ is a reflexive Banach space, then the operator equation \eqref{eq:opeq} is always ill-posed of type~II.
\end{proposition}
\begin{proof}
As a consequence of the theorem from \cite{GoldThorp63} we have that every bounded linear operator $A: \ell^1 \to Y$  is strictly singular if $Y$ is a reflexive Banach space. Hence well-posedness and ill-posedness of type~I cannot occur in such case.
\end{proof}

\begin{example}{\rm
Consider that, as mentioned before in Section~\ref{sec:2}, we have a composition $A=\tilde A \circ L$ with forward operator $\tilde A: \tilde X \to Y$ for reflexive $Y$ and synthesis operator $L:\ell^1 \to \tilde X$. Then \eqref{eq:opeq} is ill-posed of type~II even if $\tilde A$ is
continuously invertible and hence the equation $\tilde A \tilde x =y$ well-posed. This may occur, for example, for Fredholm or Volterra integral equations of the second kind. Similarly, if $\tilde A$ as mapping between Hilbert spaces is non-compact with non-closed range and hence $\tilde A \tilde x =y$ is ill-posed of type~I (which occurs, e.g., for multiplication operators mapping in $L^2(0,1)$), \eqref{eq:opeq} is still ill-posed of type~II . In the frequent case that $\tilde X$ is a separable Hilbert space and $\{u^{(k)}\}_{k \in \mathbb{N}}$ an orthonormal basis, then $A$ is compact whenever $\tilde A: \tilde X \to Y$ is compact (occurring for example for Fredholm or Volterra integral equations of the first kind).}
\end{example}

\begin{example}{\rm
If $A:=\mathcal{E}_q$ with $1 \le q < \infty$ and $Y:=\ell^q$ is the embedding operator, then solving equation \eqref{eq:opeq} based on noisy data $y^\delta \in \ell^q$ fulfilling \eqref{eq:noise} is a {\it denoising problem} (see also \cite[Sect.~5]{FleHofVes15} and \cite[Example~6.1]{FleHofVes16}). For $1<q < \infty$ the embedding operator $A=\mathcal{E}_q$  is strictly singular with non-closed range but non-compact.
     Due to Proposition~\ref{pro:notI} the equation is ill-posed of type~II. Moreover, we have $Ae^{(k)} \rightharpoonup 0$ in $\ell^q$, which due to Proposition~\ref{pro:flid1} implies that $A$ is weak$^*$-to-weak continuous and hence $\mathcal{R}(A^*) \subseteq c_0$. The latter is obvious, because the adjoint $A^*$ is the embedding operator from $\ell^{q^*}$ to $\ell^\infty$ with $1/q+1/q ^*=1$ and $\mathcal{R}(A^*)=\ell^{q^*}$. In particular, the source condition \eqref{eq:BFH}  applies with $f^{(k)}=e^{(k)} \in \ell^{q^*} \subset c_0$ for all $k \in \mathbb{N}$.}
\end{example}

\begin{example}{\rm
For $q=1$ in the previous example we have the continuously invertible identity operator $A=Id: \ell^1 \to \ell^1$ with closed range $\mathcal{R}(A)=\ell^1$. Then equation \eqref{eq:opeq} is {\it well-posed}, but we have $Ae^{(k)} \not \rightharpoonup 0$ in $\ell^1$ for $k \to \infty$, which due to Proposition~\ref{pro:flid1} indicates that the range $\mathcal{R}(A^*)$ of the adjoint of $A$ does not belong to $c_0$ and in particular $A$ is not weak$^*$-to-weak continuous. This is evident, because the adjoint of $A=Id$ is the identity $A^*=Id: \ell^\infty \to \ell^\infty$ and  $\mathcal{R}(A^*)=\ell^\infty$. We will come back to this example later.}
\end{example}

For obtaining error estimates in $\ell^1$-regularization on which convergence rates are based, we need some kind of conditional well-posedness in order to overcome the ill-posedness of equation \eqref{eq:opeq}.
Well-posed varieties of equation \eqref{eq:opeq} yield stability estimates $\|x-x^\dagger\|_{\ell^1} \le K \|Ax-Ax^\dagger\|_Y$ for all $x \in \ell^1$, which under \eqref{eq:noise} and for the choice $\alpha=\alpha_{SDP}$ imply the best possible rate
\begin{equation} \label{eq:correctrate}
\|\xad-\xdag\|_{\ell^1} = O(\delta) \quad \mbox{as} \quad \delta \to 0\,,
\end{equation}
which is typical for well-posed situations. We will come back to this in Section \ref{sec:6}.
We say that a {\it conditional stability estimate} holds true if there is a subset $\mathcal{M} \subset \ell^1$ such that
\begin{equation} \label{eq:CSE}
\|x-x^\dagger\|_{\ell^1} \le K(\mathcal{M}) \|Ax-Ax^\dagger\|_Y \quad \mbox{for all} \quad x \in \mathcal{M}\,.
\end{equation}
Because $\mathcal{M}$ is not known a priori, such kind of stability requires the additional use of regularization for bringing the approximate solutions to $\mathcal{M}$ such that a rate \eqref{eq:correctrate}
can be verified. This idea was first published in \cite{ChengYam00}
by Cheng and Yamamoto. In the context of $\ell^1$-regularization for our equation \eqref{eq:opeq}, we have estimates of the form \eqref{eq:CSE} if the solution $\xdag \in \ell^0$ is {\it sparse}, i.e.~only a finite
number of non-zero components occur in the infinite sequence $\xdag$. Then $\mathcal{M}$ can be considered as a subset of $\ell^0$ with specific properties, and the sparsity of $\ell^1$-regularized solutions
verified in Proposition~\ref{pro:prorep} ensures that the corresponding approximate solutions belong to $\mathcal{M}$. This implies the rate  \eqref{eq:correctrate} for $\xdag \in \ell^0$, although equation \eqref{eq:opeq} is not well-posed.

A similar but different kind of conditional well-posedness estimates are {\it variational source conditions}, which attain in our setting the form
\begin{equation} \label{eq:VI}
\beta\,\|x-x^\dagger\|_{\ell^1} \le \|x\|_{\ell^1}- \|\xdag\|_{\ell^1}+\varphi(\|Ax-Ax^\dagger\|_Y) \quad \mbox{for all} \quad x \in \ell^1\,,
\end{equation}
satisfied for a constant $0<\beta \le 1$ and some concave index function $\varphi$. From \cite[Theorems~1 and 2]{HofMat12} we find directly the convergence rates results of the subsequent proposition.
\begin{proposition} \label{pro:l1rates}
If the variational source condition \eqref{eq:VI} holds true for a constant $0<\beta \le 1$ and some concave index function $\varphi$, then we have for $\ell^1$-regularized solutions $\xad$ the convergence rate
\begin{equation} \label{eq:phirate}
\|\xad-\xdag\|_{\ell^1} = O(\varphi(\delta)) \quad \mbox{as} \quad \delta \to 0
\end{equation}
whenever the regularization parameter is chosen either a priori $\alpha=\alpha_{APRI}$ according to \eqref{eq:alphaapriori} or a posteriori as $\alpha=\alpha_{SDP}$ according to \eqref{eq:sdp}.
\end{proposition}
Consequently, for the manifestation of convergence rates results in the next section it remains to find constants $\beta$, concave index functions $\varphi$ and sufficient conditions for the verification of corresponding variational inequalities \eqref{eq:VI}.

\section{Convergence rates for $\ell^1$-regularization} \label{sec:4}

The first step to derive a variational source condition \eqref{eq:VI} at the solution point $\xdag=(\xdag_1,\xdag_2,...) \in \ell^1$ was taken by Lemma~5.1 in \cite{BurFleHof13}, where the inequality
\begin{equation} \label{eq:firstlemma}
\|x-x^\dagger\|_{\ell^1} \le \|x\|_{\ell^1}- \|\xdag\|_{\ell^1}+2\,\left(\sum \limits_{k=n+1}^\infty |x^\dagger_k|+\sum \limits_{k=1}^n |x_k-x^\dagger_k| \right)
\end{equation}
was proven for all $x=(x_1,x_2,...) \in \ell^1$ and all $n\in\mathbb{N}$. Then under the source condition \eqref{eq:BFH} ,valid for all $k \in \mathbb{N}$, one directly finds
\begin{equation}\label{eq:finitedim_part}
\sum \limits_{k=1}^n |x_k-x^\dagger_k| =\sum \limits_{k=1}^n \langle e^{(k)},x-\xdag\rangle_{\ell^\infty \times \ell^1}=\sum \limits_{k=1}^n\langle f^{(k)},A(x-\xdag)\rangle_{Y^* \times Y}
\end{equation}
and hence from \eqref{eq:firstlemma} that a function of type
\begin{equation} \label{eq:generalphi}
\varphi(t)=2\,\inf \limits_{n \in \mathbb{N}} \left(\sum \limits_{k=n+1}^\infty |x^\dagger_k|+\gamma_n\,t \right)
\end{equation}
with $\beta=1$ and
\begin{equation} \label{eq:gamma1}
\gamma_n= \sum \limits_{k=1}^n \|f^{(k)}\|_{Y^*}
\end{equation}
provides us with a variational inequality \eqref{eq:VI}. Along the lines of the proof of \cite[Theorem~5.2]{BurFleHof13} one can show the assertion of the following lemma.

\begin{lemma} \label{lem:concave}
If $\{\gamma_n\}_{n \in \mathbb{N}}$ is a non-decreasing sequence, then $\varphi$ from \eqref{eq:generalphi} is a well-defined and concave index function for all $\xdag \in \ell^1$.
\end{lemma}
Both the decay rate of $\xdag_k \to 0$ as $k \to \infty$ and the behaviour of $\gamma_n$ as $n \to \infty$ in \eqref{eq:generalphi} have impact on the resulting rate function $\varphi$. A power-type decay of $\xdag_k$ leads to
H\"older convergence rates (see \cite[Example~5.3]{BurFleHof13} and \cite[Example~3.4]{FleHofVes15}), whereas exponential decay of $\xdag_k$ leads to near-to-$\delta$ rates slowed down by a logarithmic factor (see \cite[Example~3.5]{BoHo13} and \cite[Example~3.5]{FleHofVes15}). In the case that $\xdag \in \ell^0$ is sparse with $\xdag_k=0$ for all $k>n_0$, then the best possible rate \eqref{eq:correctrate} is seen. This becomes clear from formula \eqref{eq:generalphi}, because then $\varphi$ fulfills the inequality $\varphi(t) \le 2 \gamma_{n_0} \,t$.

From Proposition~\ref{pro:l1rates} we have that for all concave index functions $\varphi$ from \eqref{eq:generalphi} a convergence rate \eqref{eq:phirate} for the $\ell^1$-regularization takes place in the case of appropriate choices of the regularization parameter $\alpha$  whenever a constant $0<\beta \le 1$
exists such that \eqref{eq:VI} is valid with $\varphi$ from \eqref{eq:generalphi}. When the condition \eqref{eq:BFH} is valid, this is the case with $\beta=1$ and $\gamma_n$ from \eqref{eq:gamma1}. Under the same condition the rate
was slightly improved in \cite{FleHofVes15} (see also \cite{FleHofVes16}) by showing that $\gamma_n$ from \eqref{eq:gamma1} can be replaced with
\begin{equation} \label{eq:gamma2}
\gamma_n=  \sup_{\substack{a_k\in\{-1,0,1\}\\k=1,\ldots,n}}\left\Vert\sum_{k=1}^n a_k f^{(k)}\right\Vert_{Y^*}.
\end{equation}

However, the condition \eqref{eq:BFH} may fail as was noticed first in \cite{FleHeg14} for a bidiagonal operator. Therefore, assumption \eqref{eq:BFH} was replaced by a weaker (but not particularly eye-pleasing) one in \cite{FleHeg14}.
Ibid the authors assume, in principle, that for each $n\in \mathbb{N}$ there are elements $f^{(n,k)}$ such that for all $1\leq i\leq n$
\[
[A^\ast f^{(n,k)}]_i =[e^{(k)}]_i
\]
and
\[
\left|\sum\limits_{k=1}^n [A^\ast f^{(n,k)}]_i\right|\leq c \quad \mbox{for all} \quad i>n \quad\mbox{and} \quad c<1.
\]
This means that each basis vector $e^{(k)}$ can be approximated exactly up to arbitrary position but with a non-zero tail consisting of sufficiently small elements. Later, in \cite{FleHofVes16}, a more clearly formulated property was assumed which implies the one from \cite{FleHeg14}. We give a slightly reformulated version of this property in the following. In this context, we notice that $P_n$ denotes the projection operator applied to elements $x=(x_1,x_2,...,x_n,x_{n+1},...)$ such that
$P_n x=(x_1,x_2,...,x_n,0,0,...)$.
\begin{property}\label{property_ill_posedness}
For arbitrary $\mu\in[0,1)$, we have a real sequence $\{\gamma_n\}_{n\in\mathbb{N}}$ such that for each $n\in\mathbb{N}$ and each $\xi=\xi(n)\in\ell^\infty$, with
\begin{equation}\label{eq:xi}
\xi_k\begin{cases}\in[-1,1],&\text{if }k\leq n,\\=0,&\text{if }k>n\end{cases}\,,
\end{equation}
there exists some $\eta=\eta(\mu,n,\xi) \in Y^\ast$ satisfying
\begin{itemize}
\item[(a)]\quad
$P_nA^\ast\eta=\xi$,
\item[(b)]\quad
$\vert[(I-P_n)A^\ast\eta]_k\vert\leq\mu$\quad for all $\;k>n$,
\item[(c)]\quad
$\|\eta\|_{Y^\ast}\leq\gamma_n$.
\end{itemize}
\end{property}

It is important to note that it was a substantial breakthrough in the recent paper \cite{FleGer17} to show that Property \ref{property_ill_posedness} follows directly from injectivity and weak$^*$-to-weak continuity of the operator $A$. Namely, the following proposition was proven there. Note that we changed the definition of the $\xi$ in \eqref{eq:xi} slightly. By checking the proofs in the original paper one sees the amendments we made are not relevant.
\begin{proposition}\label{prop:weakstarweak}
Let $A:\ell^1\rightarrow Y$ be bounded, linear and weak$^\ast$-to-weak continuous. Then the following assertions are equivalent.
\begin{itemize}
\item[(i)]\quad
$A$ is injective,
\item[(ii)]\quad
$e^{(k)}\in\overline{\mathcal{R}(A^\ast)}^{\ell^\infty}\;$ for all $\;\;k\in\mathbb{N}$,
\item[(iii)]\quad
$\overline{\mathcal{R}(A^\ast)}^{\ell^\infty}=c_0$,
\item[(iv)]\quad
Property \ref{property_ill_posedness} holds.
\end{itemize}
\end{proposition}
In other words, for such operators there exist appropriate sequences $\{\gamma_n\}_{n \in \mathbb{N}}$ occurring in \eqref{eq:generalphi} such that a variational source condition~\eqref{eq:VI} holds for an index function $\varphi$ from \eqref{eq:generalphi} and constant $\beta=\frac{1-\mu}{1+\mu}$ (see Proposition~\ref{thm:rates} below). Item (b) in Property~\ref{property_ill_posedness} is a generalization of \eqref{eq:BFH}. Namely, the canonical basis vectors $e^{(k)}$ do not necessarily belong to the range of $A^\ast$ but to its closure. For the proof of Proposition~\ref{prop:weakstarweak} we refer to \cite{FleGer17}. Most of the steps are identical or at least similar to the proof of
Proposition~\ref{prop:weakstarweakstar} which we will give later.

\section{Non-reflexive image spaces} \label{sec:5}
If the injective bounded linear operator $A:\ell^1 \to Y$ fails to be weak$^*$-to-weak continuous, then the results of the preceding section do not apply. In case that $Y$ is a non-reflexive Banach space, it makes sense to consider the
weaker property of weak$^*$-to-weak$^*$ continuity of $A$. An already mentioned example is the identity mapping $A=Id$ for $Y=\ell^1$. In $\ell^1$, weak convergence and norm convergence coincide (Schur property), but there is no coincidence with  weak$^*$ convergence. Thus, the identity mapping cannot be weak$^*$-to-weak continuous, but it is weak$^*$-to-weak$^*$ continuous as the following Proposition~\ref{pro:flid2} shows. It is a modified extension of Proposition~\ref{pro:flid1}.
Following \cite[Lemma~6.5]{FlemmingHabil17} we formulate this extension and repeat below the relevant proof details.

\begin{proposition} \label{pro:flid2}
Let $Z$ be a separable Banach space which acts as a predual space for the Banach space $Y=Z^\ast$.
Then the following four assertions are equivalent.
\begin{itemize}
\item[(i)]
\quad$\{Ae^{(k)}\}_{k\in\mathbb{N}}$ converges in $Y$ weakly$^*$ to zero\,,
\item[(ii)]
\quad$\mathcal{R}(A^\ast\vert_Z):=\{v \in \ell^\infty:\, v=A^*z \;\;\;\mbox{for some}\;\;\; z \in Z \subseteq Y^*\}\subseteq c_0$\,,
\item[(iii)]
\quad$A$ is weak$^*$-to-weak$^*$ continuous\,,
\item[(iv)]
\quad There is a bounded linear operator $S: Z\rightarrow c_0$ such that $A=S^\ast$\,.
\end{itemize}

\end{proposition}

\begin{proof}
Let (i) be satisfied. Then for each $A^\ast z$ from $\mathcal{R}(A^\ast\vert_Z)$ we have
\begin{equation*}
[A^\ast z]_k=\la A^\ast z,e^{(k)}\ra_{\ell^\infty\times\ell^1}=\la z,Ae^{(k)}\ra_{Y^\ast\times Y}=\la Ae^{(k)},z\ra_{Z^\ast\times Z}\to 0
\end{equation*}
as $k\to\infty$. This yields $A^\ast z\in c_0$ and hence (ii) is valid.
\par
Now let (ii) be true and take a weakly$^*$ convergent sequence $x_n \rightharpoonup^* x_0$ in $\ell^1$ as $n \to \infty$. Then $\la Ax_n,z\ra_{Z^\ast\times Z}=\la z,Ax_n\ra_{Y^\ast\times Y}=\la A^\ast z,x_n\ra_{\ell^\infty\times\ell^1}$ for all $z$ in $Z$. Because moreover $A^\ast z$ belongs to $c_0$ and $\ell^1$ is the dual of $c_0$, we may write this as $\la A^\ast z,x_n\ra_{\ell^\infty\times\ell^1}=\la x_n,A^\ast z\ra_{\ell^1\times c_0}$.
Thus,
\begin{align*}
\lim_{n \to \infty}\la Ax_n,z\ra_{Z^\ast\times Z}
&=\lim_{n \to \infty}\la x_n,A^\ast z\ra_{\ell^1\times c_0}=\la x_0,A^\ast z\ra_{\ell^1\times c_0}\\&=\la Ax_0,z\ra_{Z^\ast\times Z}\quad\text{for all }\;z\in Z,
\end{align*}
which proves condition (iii). From (iii) and the fact that $e^{(k)} \rightharpoonup^* 0$ in $\ell^1$ as $k \to \infty$ we immediately obtain (i). Finally, the equivalence between (iii) and (iv) can be found, e.g., in \cite[Theorem 3.1.11]{Meg98}.
\end{proof}
As a consequence of item (iv) in Proposition~\ref{pro:flid2}, each weak$^*$-to-weak$^*$ continuous linear operator is automatically bounded. Figure \ref{fig:ops_and_spaces} illustrates the connection between the different spaces and operators we juggle around in this section.

\begin{figure}\begin{center}
\begin{tikzpicture}
  \matrix (m) [matrix of math nodes,row sep=3em,column sep=4em,minimum width=2em]
  {
%     Z & Y & Y^\ast \\
%     c_0 & \ell^1 & \ell^\infty \\};
%  \path[-stealth]
%    (m-1-1) edge node [left] {S} (m-2-1)
%            edge [double] node [below] {dual} (m-1-2)
%    (m-2-1) edge [double] node [below] {dual} (m-2-2)
%    (m-2-2) edge node [left] {$A=S^\ast$} (m-1-2)
%    		edge [double] node [below] {dual} (m-2-3)
%    (m-1-2) edge [double] node [below] {dual} (m-1-3)
%    (m-1-3) edge node [left] {$A^\ast=S^{\ast\ast}$} (m-2-3);
     Z \subseteq Y^\ast & Y=Z^\ast & Y^\ast=Z^{\ast\ast} \\
     c_0 & \ell^1=(c_0)^\ast & \ell^\infty=(\ell^1)^\ast=(c_0)^{\ast\ast} \\};
  \path[-stealth]
    (m-1-1) edge node [left] {S} (m-2-1)
    (m-2-2) edge node [left] {$A=S^\ast$} (m-1-2)
    (m-1-3) edge node [left] {$A^\ast=S^{\ast\ast}$} (m-2-3);
\end{tikzpicture}
\end{center}\caption{Schematic display of the operators and underlying spaces needed in this section. Top and bottom row: the (separable) Banach spaces under consideration. Middle row: the operators we work with.}\label{fig:ops_and_spaces}
\end{figure}
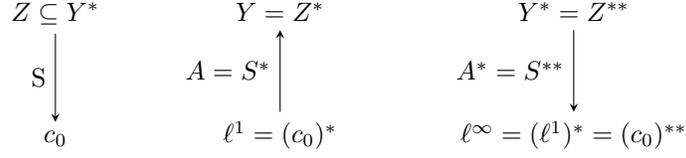

For the identity mapping $A=Id: \ell^1 \to Y$ with $Y=\ell^1$ and predual $Z=c_0$, property (i) of Proposition~\ref{pro:flid2} is trivially satisfied which yields the weak$^*$-to-weak$^*$ continuity of this operator.
Note that the case $Y=\ell^1$, $A=Id$ is only of theoretical interest. Precisely, it is a tool for exploring the frontiers of the theoretic framework we have chosen for investigating $\ell^1$-regularization. For practical applications it is irrelevant because one easily verifies that with the choice $p=1$ in \eqref{eq:TIK}, where we have $Y=\ell^1$, the $\ell^1$-regularized solutions coincide with the data $y^\delta$ if $\alpha<1$ and we have the best possible rate \eqref{eq:correctrate}.

Main parts of the above mentioned Proposition~\ref{pro:prorep} on existence, stability and convergence of $\ell^1$-regularized solutions $\xad$ remain true if the operator $A: \ell^1 \to Y$ is only weak$^*$-to-weak$^*$ continuous.
The sparsity property $\xad \in \ell^0$, however, will fail in general (consider the example of the identity as mentioned above). Existence, stability and convergence assertions remain valid, because their proofs basically rely on the fact that the mapping $x \mapsto \|Ax-y^\delta\|_Y$ is a weakly$^*$ lower semicontinuous functional. This is the case in both variants, with or without $^*$, since the norm functional is weakly and also weakly$^*$ lower semicontinuous. For the existence of regularized solutions
(minimizers of the Tikhonov functional \eqref{eq:TIK}) again the Banach-Alaoglu theorem (Lemma \ref{lem:Alaoglu}) is required and yields weakly$^*$ compact level sets of the $\ell^1$-norm functional.

Our goal is to proof an analogue to Proposition \ref{prop:weakstarweak} for weak$^*$-to-weak$^*$ continuous operators. We start with a first observation.

\begin{proposition}\label{prop:ra}
Let $A:\ell^1\rightarrow Y$ be injective and weak$^*$-to-weak$^*$ continuous and let $Y=Z^\ast$ for some Banach space $Z$. Then
\begin{equation*}
\overline{\mathcal{R}(A^\ast\vert_Z)}^{\ell^\infty}=c_0.
\end{equation*}
\end{proposition}

\begin{proof}
From item (iv) of Proposition~\ref{pro:flid2} we take the operator $S:Z\rightarrow c_0$ with $A=S^\ast$. As $A$ is injective, i.e., $\mathcal{N}(A)=\{ 0 \}$, it follows 
\[
\overline{\mathcal{R}(S)}^{c_0}=\mathcal{N}(S^\ast)_\perp=\mathcal{N}(A)_\perp=c_0.
\]
There, the subscript denotes the pre-annihilator for a set $V$, in our situation with $(c_0)^\ast=\ell^1$ and $V\subseteq \ell^1$ defined as
\[
V_\perp:=\{ x\in c_0: \langle \zeta, x\rangle_{\ell^1\times c_0}=0 \quad \forall \zeta\in V\}.
\]
Let $\eta\in Z$ and recall $Y^\ast=Z^{\ast\ast}$ (cf. Figure \ref{fig:ops_and_spaces}). Then for each $x\in\ell^1$
\begin{align*}
\la A^{\ast}\,\eta,x \ra_{\ell^\infty\times \ell^1}
&=\la \eta,A\,x\ra_{Y^\ast \times Y}
=\la \eta,A\,x\ra_{Z\times Y}
=\la S\,\eta,x\ra_{c_0\times \ell^1}\\
&=\la S\,\eta,x\ra_{\ell^\infty\times \ell^1},
\end{align*}
i.e., $A^{\ast}\vert_{Z}=S$. Thus  $\overline{\mathcal{R}(A^\ast\vert_{Z})}^{\ell^\infty}=\overline{\mathcal{R}(S)}^{c_0}=c_0$. At this point we emphasize that in both Banach spaces $\ell^\infty$ and $c_0$ the same
supremum norm applies.
\end{proof}
We will show in Proposition \ref{prop:weakstarweakstar} that conversely $\overline{\mathcal{R}(A^\ast\vert_Z)}^{\ell^\infty}=c_0$ implies injectivity for weak$^*$-to-weak$^*$ continuous operators. Before doing so we need the following Proposition which coincides in principle with \cite[Proposition~9]{FleGer17}.
\begin{proposition}\label{th:ra}
Let $A$ be injective and weak$^*$-to-weak$^*$ continuous. Moreover, let $\varepsilon>0$ and $n\in\mathbb{N}$. Then for each $\xi\in c_0$ there exists $\tilde{\xi}\in\mathcal{R}(A^\ast)$ such that
\begin{equation*}
\tilde{\xi}_k=\xi_k\quad\text{for $\;\;k\leq n$}\qquad\text{and}\qquad\vert\tilde{\xi}_k-\xi_k\vert\leq\varepsilon\quad\text{for $\;\;k>n$}.
\end{equation*}
\end{proposition}

\begin{proof}
We proof the proposition by induction with respect to $n$. For $\xi\in c_0$ set
\begin{equation*}
\xi^+:=(\xi_1+\varepsilon,\xi_2,\xi_3,\ldots)\qquad\text{and}\qquad\xi^-:=(\xi_1-\varepsilon,\xi_2,\xi_3,\ldots).
\end{equation*}
By Proposition~\ref{prop:ra} we have that $c_0=\overline{\mathcal{R}(A^\ast\vert_Z)}^{\ell^\infty}\subset \overline{\mathcal{R}(A^\ast)}^{\ell^\infty}$. Hence we find elements $\tilde{\xi}^+\in\mathcal{R}(A^\ast)$ and $\tilde{\xi}^-\in\mathcal{R}(A^\ast)$ with
\begin{equation*}
\|\tilde{\xi}^+-\xi^+\|_{\ell^\infty}\leq\varepsilon\qquad\text{and}\qquad\|\tilde{\xi}^--\xi^-\|_{\ell^\infty}\leq\varepsilon.
\end{equation*}
Consequently, $\tilde{\xi}^+_1\geq\xi_1\geq\tilde{\xi}^-_1$ and $\vert\tilde{\xi}^+_k-\xi_k\vert\leq\varepsilon$ as well as $\vert\tilde{\xi}^-_k-\xi_k\vert\leq\varepsilon$ for $k>1$. Thus we find a convex combination $\tilde{\xi}$ of $\tilde{\xi}^+$ and $\tilde{\xi}^-$ such that $\tilde{\xi}_1=\xi_1$. This $\tilde{\xi}$ obviously also satisfies $\vert\tilde{\xi}_k-\xi_k\vert\leq\varepsilon$ for $k>1$, which proves the proposition for $n=1$.
\par
Now let the proposition be true for $n=m$. We prove it for $n=m+1$.
Let $\xi\in c_0$ and set
\begin{align*}
\xi^+&:=(\xi_1,\ldots,\xi_m,\xi_{m+1}+\varepsilon,\xi_{m+2},\xi_{m+3},\ldots),\\
\xi^-&:=(\xi_1,\ldots,\xi_m,\xi_{m+1}-\varepsilon,\xi_{m+2},\xi_{m+3},\ldots).
\end{align*}
By the induction hypothesis we find $\tilde{\xi}^+\in\mathcal{R}(A^\ast)$ and $\tilde{\xi}^-\in\mathcal{R}(A^\ast)$ with
\begin{equation*}
\tilde{\xi}^+_k=\xi_k=\tilde{\xi}^-_k\quad\text{for $k\leq m$}
\end{equation*}
and
\begin{equation*}
\vert\tilde{\xi}^+_k-\xi^+_k\vert\leq\varepsilon\quad\text{and}\quad\vert\tilde{\xi}^-_k-\xi^-_k\vert\leq\varepsilon\quad\text{for $k>m$}.
\end{equation*}
Consequently, we have $\tilde{\xi}^+_{m+1}\geq\xi_{m+1}\geq\tilde{\xi}^-_{m+1}$ and $\vert\tilde{\xi}^+_k-\xi_k\vert\leq\varepsilon$ as well as $\vert\tilde{\xi}^-_k-\xi_k\vert\leq\varepsilon$ for $k>m+1$. Thus we find a convex combination $\tilde{\xi}$ of $\tilde{\xi}^+$ and $\tilde{\xi}^-$ such that $\tilde{\xi}_{m+1}=\xi_{m+1}$. This $\tilde{\xi}$ obviously also satisfies $\tilde{\xi}_k=\xi_k$ for $k<m+1$ and $\vert\tilde{\xi}_k-\xi_k\vert\leq\varepsilon$ for $k>m+1$, which proves the proposition for $n=m+1$.
\end{proof}

Now we come to the main result of this section. The proof is similar and in part identical to the one of Proposition 12 in \cite{FleGer17}.

\begin{proposition}\label{prop:weakstarweakstar}
Let $A:\ell^1\rightarrow Y$ be bounded, linear and weak$^\ast$-to-weak$^\ast$ continuous. Then the following assertions are equivalent.
\begin{itemize}
\item[(i)]\quad
$A$ is injective,
\item[(ii)]\quad
$\overline{\mathcal{R}(A^\ast\vert_Z)}^{\ell^\infty}=c_0$,
\item[(iii)]\quad
$e^{(k)}\in\overline{\mathcal{R}(A^\ast\vert_Z)}^{\ell^\infty}\;\; $ for all $\;\;k\in\mathbb{N}$,
\item[(iv)]\quad
Property \ref{property_ill_posedness} holds.
\end{itemize}
\end{proposition}
\begin{proof}
We show (i) $\Rightarrow$ (iv) $\Rightarrow$ (iii) $\Rightarrow$ (ii) $\Rightarrow$ (i).
\par
(i)$\Rightarrow$(iv): Fix $\mu\in(0,1)$, $n\in\mathbb{N}$ and take some $\xi$ as described in Property \ref{property_ill_posedness}.
By Proposition~\ref{th:ra} with $\varepsilon:=\mu$ there exists some $\eta$ such that $A^\ast\eta$ ($=\tilde{\xi}$ in the proposition) satisfies items (a) and (b) in Property \ref{property_ill_posedness}. In particular we have $\{\eta_k\}_{k\in{1,\dots,n}}$ such that $P_nA^\ast\eta_k=e^{(k)}$ and $|[(I-P_n)A^\ast\eta_k]_i|\leq \frac{\mu}{n}$ for all $i>n$. Since $\xi\in c_0$ it is
\[
\xi=\sum \limits_{k=1}^n c_k e^{(k)}=\sum \limits_{k=1}^n c_k A^\ast\eta_k=A^\ast\left(\sum \limits_{k=1}^n c_k \eta_k\right),
\]
for coefficients $-1\leq c_k\leq 1$, i.e., $\xi=A^\ast\eta$ with $||\eta||\leq \sum \limits_{k=1}^n ||\eta_k||$ as an upper bound for $\gamma_n$. By construction this $\eta$ also fulfills 
\[
|[(I-P_n)A^\ast\eta]_i|\leq \sum\limits_{i=1}^n |[(I-P_n)A^\ast\eta_k]_i|\leq\mu.
\]
\par
(iv)$\Rightarrow$(iii): Fix $k$, fix $n\geq k$, take a sequence $(\mu_m)_{m\in\mathbb{N}}$ in $(0,1)$ with $\mu_m\to 0$ and choose $\xi:=e^{(k)}$ in Property \ref{property_ill_posedness}. Then for a corresponding sequence $(\eta_m)_{m\in\mathbb{N}}$ from Property \ref{property_ill_posedness} we obtain
\begin{equation*}
\|e^{(k)}-A^\ast\eta_m\|_{\ell^\infty}
\leq\|e^{(k)}-P_nA^\ast\eta_m\|_{\ell^\infty}+\|(I-P_n)A^\ast\eta_m\|_{\ell^\infty}.
\end{equation*}
The first summand is zero by the choice of $\xi$ and the second summand is bounded by $\mu_m$. Thus, $\|e^{(k)}-A^\ast\eta_m\|_{\ell^\infty}\to 0$ if $m\to\infty$.
\par
(iii)$\Rightarrow$(ii):
$(e^{(k)})_{k\in\mathbb{N}}$ is a Schauder basis in $c_0$. Thus, $c_0\subseteq\overline{\mathcal{R}(A^\ast\vert_Z)}^{\ell^\infty}$. Proposition \ref{pro:flid2} yields $\overline{\mathcal{R}(A^\ast\vert_Z)}^{\ell^\infty}\subseteq c_0$. Hence $\overline{\mathcal{R}(A^\ast\vert_Z)}=c_0$.
\par
(ii)$\Rightarrow$(i):
One easily shows $\overline{\mathcal{R}(A^\ast)}^{\ell^\infty}\subseteq\mathcal{N}(A)^\perp$. Thus $c_0\subseteq\mathcal{N}(A)^\perp$. If we have $x\in\ell^1$ with $Ax=0$, then for each $u\in c_0\subseteq\mathcal{N}(A)^\perp$ we obtain
\begin{equation*}
\la x,u\ra_{\ell^1\times c_0}=\la u,x\ra_{\ell^\infty\times\ell^1}=0,
\end{equation*}
which is equivalent to $x=0$.
\end{proof}
Since, in the context of both Propositions~\ref{prop:weakstarweak} and \ref{prop:weakstarweakstar}, the injectivity of $A$ yields Property~\ref{property_ill_posedness}, the consequences with respect to variational source conditions and convergence rate results are identical for a weak$^*$-to-weak and a weak$^*$-to-weak$^*$ continuous operator $A$. We formulate the following theorem and the subsequent corollary and prove the theorem for a weak$^*$-to-weak$^*$ continuous
 operator $A:\ell^1 \to Y$. In particular, the corollary requires the existence of a separable predual space $Z$ of $Y$ in order to apply Lemma~\ref{lem:Alaoglu} and to ensure the stabilizing property of the Tikhonov penalty. However, the proof of the theorem repeats point by point the ideas of the proof from \cite[Corollary 11]{FleGer17} focused on weak$^*$-to-weak continuous operators $A$.

\begin{theorem}\label{thm:rates}
Let the bounded linear operator $A: \ell^1 \to Y$ be injective and weak$^*$-to-weak$^*$ continuous, and assume additionally that the Banach space $Y$ possesses a separable predual Banach space $Z$ with $Z^*=Y$. Moreover, let $\mu \in [0,1)$ and $\{\gamma_n\}_{n\in\mathbb{N}}$ be such that Property \ref{property_ill_posedness} is fulfilled. Then a variational source condition \eqref{eq:VI} with the constant $\beta=\frac{1-\mu}{1+\mu} \in [0,1)$ and the concave index function $\varphi$ given by \eqref{eq:generalphi} is fulfilled.
\end{theorem}

\begin{proof}
Fix $n\in\mathbb{N}$ and $x\in\ell^1$ and let $\xi:=\mathrm{sgn}\, P_n(x-x^\dagger)\in\ell^\infty$ be the sequence of signs of $P_n(x-x^\dagger)$.
Then by Property \ref{property_ill_posedness} there is some $\eta$ such that
\begin{align}\label{eq:finitedim_condfour}
\lefteqn{\sum \limits_{k=1}^n |x_k-x^\dagger_k|\nonumber
=\langle \xi, x-x^\dag\rangle_{\ell^\infty\times\ell^1}= \langle P_n A^\ast \eta, x-x^\dag\rangle_{\ell^\infty\times\ell^1}}\nonumber\\
&\qquad=\langle P_n A^\ast \eta-A^\ast\eta, x-x^\dag\rangle_{\ell^\infty\times\ell^1}+\langle  A^\ast \eta, x-x^\dag\rangle_{\ell^\infty\times\ell^1}\nonumber\\
&\qquad=-\langle (I-P_n)A^\ast\eta, (I-P_n)(x-x^\dag)\rangle_{\ell^\infty\times\ell^1}+\langle  A^\ast \eta, x-x^\dag\rangle_{\ell^\infty\times\ell^1}\nonumber\\
&\qquad\leq \mu\|(I-P_n)(x-x^\dag)\|_{\ell^1}+\gamma_n\|Ax-Ax^\dag\|_Y.
\end{align}
The triangle inequality yields
\begin{equation}\label{eq:pm}
\|P_n(x-x^\dagger)\|_{\ell^1}\leq \mu\bigl(\|(I-P_n)x\|_{\ell^1}+\|(I-P_n)x^\dagger\|_{\ell^1}\bigr)+\gamma_n\|Ax-Ax^\dagger\|_Y.
\end{equation}
Now
\begin{align*}
\lefteqn{\beta\|x-x^\dagger\|_{\ell^1}-\|x\|_{\ell^1}+\|x^\dagger\|_{\ell^1}}\\
&=\beta\|P_n(x-x^\dagger)\|_{\ell^1}+\beta\|(I-P_n)(x-x^\dagger)\|_{\ell^1}-\|P_nx\|_{\ell^1}-\|(I-P_n)x\|_{\ell^1}\\
&\quad+\|P_nx^\dagger\|_{\ell^1}+\|(I-P_n)x^\dagger\|_{\ell^1}
\end{align*}
together with
\begin{equation*}
\beta\|(I-P_n)(x-x^\dagger)\|_{\ell^1}\leq\beta\|(I-P_n)x\|_{\ell^1}+\beta\|(I-P_n)x^\dagger\|_{\ell^1}
\end{equation*}
and
\begin{equation*}
\|P_nx^\dagger\|_{\ell^1}=\|P_n(x-x^\dagger-x)\|_{\ell^1}\leq\|P_n(x-x^\dagger)\|_{\ell^1}+\|P_nx\|_{\ell^1}
\end{equation*}
shows
\begin{align*}
\lefteqn{\beta\|x-x^\dagger\|_{\ell^1}-\|x\|_{\ell^1}+\|x^\dagger\|_{\ell^1}}\\
&\qquad\leq 2\|(I-P_n)x^\dagger\|_{\ell^1}+(1+\beta)\|P_n(x-x^\dagger)\|_{\ell^1}\\
&\qquad\quad\,-(1-\beta)\bigl(\|(I-P_n)x\|_{\ell^1}+\|(I-P_n)x^\dagger\|_{\ell^1}\bigr).
\end{align*}
Combining this estimate with the previous estimate \eqref{eq:pm} and taking into account that $\beta=\frac{1-\mu}{1+\mu}$
and $\mu=\frac{1-\beta}{1+\beta}$ we obtain for all $x \in \ell^1$
\begin{align*}
\beta\|x-x^\dagger\|_{\ell^1}-\|x\|_{\ell^1}+\|x^\dagger\|_{\ell^1}
&\leq 2\|(I-P_n)x^\dagger\|_{\ell^1}+\frac{2}{1+\mu}\,\gamma_n\|Ax-Ax^\dagger\|_Y\\
&\leq 2\|(I-P_n)x^\dagger\|_{\ell^1}+2\gamma_n\|Ax-Ax^\dagger\|_Y.
\end{align*}
Taking the infimum over all $n\in\mathbb{N}$ completes the proof.
\end{proof}
The variational source condition immediately yields convergence rates according to Proposition \ref{pro:l1rates}.
\begin{corollary}
Under the conditions of Theorem~\ref{thm:rates} the $\ell^1$-regularized solutions $\xad$ as minimizers of \eqref{eq:TIK} fulfil
\begin{equation*}
\|\xad-\xdag\|_{\ell^1} = O(\varphi(\delta)) \quad \mbox{as} \quad \delta \to 0\,,
\end{equation*}
with the concave index function $\varphi$ from \eqref{eq:generalphi}, whenever the regularization parameter $\alpha$ is chosen either a priori as $\alpha=\alpha_{APRI}$ according to \eqref{eq:alphaapriori} or a posteriori as $\alpha=\alpha_{SDP}$ according to \eqref{eq:sdp}.
\end{corollary}

In order to familiarize the reader with the concepts in this work we will look at a particular operator to exemplify our theory. In particular we verify Property \ref{property_ill_posedness}.

\begin{example}{\rm
Let $X=Y=\ell^1$ and
\[
[Ax]_k=x_k+x_{k+1}\qquad k\in\mathbb{N},
\]
i.e., $A$ maps $x=(x_1,x_2,x_3,\dots)$ to $Ax=(x_1+x_2,x_2+x_3,x_3+x_4,\dots)$. Clearly $A$ is linear. Observe that $||Ax||_{\ell^1}\leq 2||x||_{\ell^1}$ and hence $A$ is bounded. One easily verifies the adjoint $A^\ast:\ell^\infty\rightarrow\ell^\infty$,
\[
A^\ast y=(y_1,y_2+y_1,y_3+y_2,y_3+y_4,\dots).
\]
Both $A$ and $A^\ast$ are injective. Since $\overline{\mathcal{R}(A)}^{\ell^1}=\mathcal{N}(A^\ast)_\perp$, where
\[
\mathcal{N}(A^\ast)_\perp:=\{ x\in\ell^1:\langle y,x\rangle_{\ell^\infty\times\ell^1}=0\;\, \forall y\in\mathcal{N}(A^\ast)\}=\ell^1
\]
we have $\overline{\mathcal{R}(A)}^{\ell^1}=\ell^1$. It is however easy to see that $\mathcal{R}(A)\neq \ell^1$. For example there is no $x\in\ell^1$ such that $Ax=e^{(2)}$. Namely, solving $Ax=e^{(2)}$ for $x$ leads to the system of equations $x_1=-x_2$, $x_2=1-x_3$, $x_4=-x_3$, $x_5=-x_4$, etc. Due to the alternating character again there is no $x\in\ell^1$ that satisfies this system. We have shown that $\overline{\mathcal{R}(A)}^{\ell^1}\neq\mathcal{R}(A)$, i.e., the corresponding operator equation \eqref{eq:opeq} is ill-posed.

Next we prove that $A$ is weak$^\ast$-to-weak$^\ast$ continuous but not weak$^\ast$-to-weak continuous. To this end we use the properties (i) in Proposition \ref{pro:flid1} and Proposition \ref{pro:flid2}, respectively. First let $\xi\in c_0$. Then, with
\[
[Ae^{(k)}]_i=\begin{cases} 1 & i=k,k-1\\ 0 & else\end{cases} \qquad \forall i\geq 2
\]
it is
\[
\langle \xi,Ae^{(k)}\rangle_{c_0\times\ell^1}= \xi_{k-1}+\xi_k\rightarrow 0\,,
\]
since $\xi\in c_0.$ For $\xi\in\ell^\infty$, however,
\[
\langle \xi,Ae^{(k)}\rangle_{\ell^\infty\times\ell^1}= \xi_{k-1}+\xi_k
\]
does in general not converge to zero (let, e.g., $\xi\equiv 1$). Summarizing the properties of the forward operator for the present example, we note that $A$ is linear, injective, weak$^\ast$-to-weak$^\ast$ continuous, but not weak$^\ast$-to-weak continuous. Moreover its range
is not closed such that the corresponding operator equation \eqref{eq:opeq} is ill-posed.

For this particular operator let us investigate Property \ref{property_ill_posedness}. We will see in the following that this actually holds with $\mu=0$ and $\gamma_n=n$, i.e., $e^{(k)}\in\mathcal{R}(A^\ast)$ for all $k\in\mathbb{N}$. We will also show that $e^{(k)}\in\overline{\mathcal{R}(A^\ast\vert_{c_0})}^{\ell^\infty}$ according to item (iii) of Proposition \ref{prop:weakstarweakstar}. Even then we still have $\gamma_n=n$ for arbitrary $0<\mu<1$.

Fix an arbitrary $n\in\mathbb{N}$ and let $\xi=(\xi_1,\xi_2,\dots,\xi_n,0,\dots)\in\ell^\infty$ with $\xi_i\in[-1,1]$. We are looking for an $\eta\in\ell^\infty$ satisfying Property \ref{property_ill_posedness}. Observe that, by definition of $A^\ast$ and for given $\xi$, any $\eta$ satisfying $P_nA^\ast \eta=\xi$ has the structure $\eta_1=\xi_1$, $\eta_2=\xi_2-\eta_1$, $\eta_3=\xi_3-\eta_2$ and so on until $\eta_n=\xi_n-\eta_{n-1}$. In other words it is $\eta_n=\sum_{i=1}^n (-1)^{n-i} \xi_i$ and
\[
||\eta_n||_{\ell^\infty}\leq n,
\]
which yields item (c) of Property \ref{property_ill_posedness} with $\gamma_n= n$. Now fix arbitrary $0<\mu<1$. We have $[A^\ast\eta]_{n+1}=\eta_{n+1}+\eta_n$ and require $|\eta_{n+1}+\eta_n|\leq\mu$. Thus we may take any $\eta_{n+1}$ with
\[
-\eta_n-\mu\leq \eta_{n+1}\leq -\eta_n+\mu.
\]
Analogously we find that in general
\[
(-1)^{i}\eta_n-i\mu\leq \eta_{n+i}\leq (-1)^{i} \eta_n+i\mu, \quad i=1,2,\dots.
\]
Therefore, the choice of the tail of $\eta$ is ambiguous. For example, a viable pick is $\eta_{n+i}=(-1)^{i}\eta_n$. Then
\begin{equation}\label{eq:eta}
\eta=(\eta_1, \eta_2,\dots, \eta_n, -\eta_n,\eta_n,-\eta_n,\dots)
\end{equation}
with $\eta_i$, $1\leq i\leq n$, as above and
\[
A^\ast\eta=(\xi_1,\xi_2,\dots,\xi_n,0,0,0,\dots).
\]
In particular, this means that $e^{(k)}\in\mathcal{R}(A^\ast)$ (choose $\xi_i=1$ and $\xi_j=0$ for $i\neq j$). Note that $\eta\in\ell^\infty$ but $\eta\notin c_0$ in \eqref{eq:eta}. However, we also see that for any $\xi$ and arbitrary $0<\mu<1$ there are (infinitely many) choices for the tail of $\eta$ such that $\eta\in c_0$ and item (b) of Property \ref{property_ill_posedness} holds. Independent of $\mu$, all choices satisfy  item (c) of Property \ref{property_ill_posedness} with $\gamma_n=n$. To set this into relation, we would obtain the same $\gamma_n$ for a diagonal operator $\tilde A:\ell^2\rightarrow\ell^2$ with singular values decaying as $\sigma_i\sim \frac{1}{\sqrt{i}}$, see \cite{Ger17}.}
\end{example}
Please note that in practice it is not necessary to verify Property \ref{property_ill_posedness} in the way we did here. In particular the sequential discrepancy principle \eqref{eq:sdp} does not require the knowledge of any of the parameters from Property \ref{property_ill_posedness} in order to guarantee the convergence rates implied by that property.

For the sake of completeness we mention that there exist bounded linear operators which are not even weak$^*$-to-weak$^*$ continuous.

\begin{example}{\rm
If $Y=\ell^1$ and
\begin{align*}
[Ax]_k:=\begin{cases}\sum\limits_{l=1}^\infty x_l,&\text{if }k=1,\\x_k,&\text{else},\end{cases}
\end{align*}
for all $k$ in $\mathbb{N}$ and all $x$ in $\ell^1$, then $Ae^{(k)}=e^{(1)}+e^{(k)}$ if $k>1$. Thus, $Ae^{(k)}\rightharpoonup^\ast e^{(1)}\neq 0$.
The same operator $A$ considered as mapping into $Y=\ell^2$ is an example of a not weak$^*$-to-weak continuous bounded linear operator in the classical Hilbert space setting for $\ell^1$-regularization. Note that, because of the first component, $A$ does not have a bounded extension to any $\ell^p$-space with $1< p<\infty$.}
\end{example}

\section{Well-posedness and further discussions} \label{sec:6}

\begin{proposition} \label{pro:well}
If the problem \eqref{eq:opeq} is well-posed, i.e. $\mathcal{R}(A) = \overline{\mathcal{R}(A)}^Y$, then under the conditions of Theorem~\ref{thm:rates} the $\ell^1$-regularized solutions $\xad$ as minimizers of \eqref{eq:TIK} fulfil
\begin{equation*}
\|\xad-\xdag\|_{\ell^1} = O(\delta) \quad \mbox{as} \quad \delta \to 0
\end{equation*}
whenever the regularization parameter $\alpha>0$ is chosen either a priori as \linebreak $\alpha=\alpha_{APRI}\sim \delta^{p-1}$ or a posteriori as $\alpha=\alpha_{SDP}$ according to the sequential discrepancy principle \eqref{eq:sdp}.
\end{proposition}
\begin{proof}
The condition $\mathcal{R}(A) = \overline{\mathcal{R}(A)}^Y$ implies $\mathcal{R}(A^*) = \overline{\mathcal{R}(A^*)}^{\ell^\infty}$ (cf.~\cite[Theorem 3.1.21]{Meg98}) and hence $V:=\mathcal{R}(A^*)$ is a closed
subspace of $\ell^\infty$, which can be considered as a Banach space with the same supremum norm as  $\ell^\infty$. Then, for the injective operator $A:\ell^1 \to Y$, Banach's theorem concerning the continuity of the inverse operator ensures that the linear operator $(A^*)^{-1}: V \to Y^*$ is bounded. Moreover, the elements $\tilde \xi \in \mathcal{R}(A^*)$ in Proposition~\ref{th:ra} associated to $\xi$ from Property~\ref{property_ill_posedness} satisfy the inequality
$\|\tilde \xi\|_{\ell^\infty} \le 1+\varepsilon$, and with $\tilde \xi=A^*\eta $ we have 
\[
\|\eta\|_{Y^*} \le \|(A^*)^{-1}\|_{V \to Y^*}\,(1+\varepsilon) \le K <\infty.
\]
Hence, the sequence $\{\gamma_n\}_{n \in \mathbb{N}}$
in Property~\ref{property_ill_posedness} is uniformly bounded by the finite positive constant $K$. Taking into account the proof of Theorem~\ref{thm:rates} we have with $\beta=\frac{1-\mu}{1+\mu}$ and for all $x \in \ell^1$
\begin{align*}
\beta\|x-x^\dagger\|_{\ell^1} \le \|x\|_{\ell^1}-\|x^\dagger\|_{\ell^1} + 2\|(I-P_n)x^\dagger\|_{\ell^1}+2\,K\,\|Ax-Ax^\dagger\|_Y,
\end{align*}
i.e.~ the variational inequality \eqref{eq:VI} with
$$\varphi(t)=2\,\inf \limits_{n \in \mathbb{N}} \left(\sum \limits_{k=n+1}^\infty |x^\dagger_k|+K\,t \right)= K\,t.$$
This, however, yields by Proposition~\ref{pro:l1rates} the rate \eqref{eq:correctrate}  and completes the proof of the proposition.
\end{proof}

Property \ref{property_ill_posedness} enables us to show convergence rates for $\ell^1$-regularization for ill-posed and well-posed problems with sparse and non-sparse solutions. It has been shown in \cite{Ger17} that the rate function $\varphi$ in \eqref{eq:phirate} does in general not saturate. Even more, the rate is always obtainable either with an a priori or an a posteriori choice of the regularization parameter. This stands in sharp contrast to classical Tikhonov regularization, i.e., \eqref{eq:TIK} with $p=2$ and $||x||_{\ell^2}$ as penalty, which is known to admit convergence rates up to a maximum of $\delta^{2/3}$ for a suitable a priori choice of the regularization parameter and only a rate of $\delta^{1/2}$ under the discrepancy principle. Since the smoothness of the solution is typically unknown, this makes $\ell^1$-regularization more attractive from the viewpoint of regularization theory than its $\ell^2$ counterpart. One simply uses the discrepancy principle and no longer has to care about the smoothness of the solution. However, one may run into trouble when the solution does not belong to $\ell^1$ but only to $\ell^2\backslash \ell^1$ such that $||x^\dag||_{\ell^1}=\infty$. In such a case we call the regularization method \eqref{eq:TIK} \textit{oversmoothing}.

There are promising results showing that even in the situation of oversmoothing, $\ell^1$-regularization may lead to convergence rates in a weaker norm. Again, an a priori choice or the discrepancy principle for the choice of $\alpha$ would lead to the optimal rates. Preliminary results have been shown in the preprint \cite{Ger16}. There, a strategy is shown to derive convergence rates in the $\ell^2$-norm for $\ell^1$-regularization for every $x^\dag\in \ell^2$. The proof of a theorem analogously to Proposition \ref{pro:l1rates} unfortunately was incomplete. It revolves around approximating $x^\dag$ with $P_n x^\dag$ and letting $n=n(\delta)\rightarrow \infty$ as $\delta\rightarrow 0$ with a specific choice of $n=n(\delta)$. The open problem was to show that the support of the approximate solutions is not larger than this $n(\delta)$. It appears that such a statement is possible by using item (c) in
Property~\ref{property_ill_posedness} and the necessary optimality condition for a minimizer of \eqref{eq:TIK}, where the latter provides us with the norm $||\eta||_{Y^*}$ in Property 1 corresponding to a $\xi=A^\ast\eta \in\partial||x_\alpha^\delta||_{\ell^1}$. Since we are not able to use a variational source condition when $||x^\dag||_{\ell^1}=\infty$ we need to use a different approach to show convergence rates. To this end we seem to have a chance to adapt the strategy of \cite{Ger17} based on elementary steps. Consequently, we hope to complete the detailed proof in an upcoming paper \cite{Ger17b}.

\section*{Acknowledgements}
The authors were financially supported by Deutsche Forschungsgemeinschaft (DFG-grant HO 1454/10-1).

%
%\section*{Appendix}
%\addcontentsline{toc}{section}{Appendix}
%
%
%When placed at the end of a chapter or contribution (as opposed to at the end of the book), the numbering of tables, figures, and equations in the appendix section continues on from that in the main text. Hence please \textit{do %not} use the \verb|appendix| command when writing an appendix at the end of your chapter or contribution. If there is only one the appendix is designated ``Appendix'', or ``Appendix 1'', or ``Appendix 2'', etc. if there is more %than one.

%\begin{equation}
%a \times b = c
%\end{equation}

\bibliographystyle{plain}
\bibliography{Gerth_Hofmann.bib}

\end{document}